\documentclass{elsarticle}

\usepackage{amssymb}

\usepackage{amssymb,amsthm} 
\usepackage{amscd,amssymb,graphics}
\usepackage[hidelinks]{hyperref}
\usepackage{amsfonts}
\usepackage{amsmath}
\usepackage{amsxtra}
\usepackage{latexsym}
\usepackage[mathcal]{eucal}
\usepackage{enumerate}
\usepackage{ifsym}
\usepackage{yfonts}
\usepackage{calc}
\input xy
\xyoption{all}
\usepackage{MnSymbol}

\newtheorem{thm}{Theorem}[section]
\newtheorem*{thm*}{Theorem}
\newtheorem{corol}[thm]{Corollary}
\newtheorem{claim}[thm]{Claim}
\newtheorem{lemma}[thm]{Lemma}
\newtheorem{prop}[thm]{Proposition}

\theoremstyle{definition}
\newtheorem{defi}[thm]{Definition}
\theoremstyle{remark}

\newtheorem{example}[thm]{Example}
\newtheorem{question}[thm]{Question}

\numberwithin{equation}{section}

\def \Z {{\Bbb Z}}

\def \seq {\subseteq}

\def \copt {\tau_{co}}
\def \eps {\varepsilon}

\def\N {{\mathbb N}}
\def\Z {{\mathbb Z}}

\def\Aut{{\mathrm {Aut}}}
\def\St{{\mathrm {St}}\,}
\def\H{{\mathrm{H}}\,}

\def\co{{\,{:}\,}}

\usepackage{lipsum}
\makeatletter
\def\ps@pprintTitle{%
 \let\@oddhead\@empty
 \let\@evenhead\@empty
 \def\@oddfoot{}%
 \let\@evenfoot\@oddfoot}
\makeatother

\begin{document}

\begin{frontmatter}



\title{Order and minimality of some topological groups \tnoteref{1}}

\author{Michael Megrelishvili}
\address{Department of Mathematics,
	Bar-Ilan University, 52900 Ramat-Gan, Israel}
\ead{megereli@math.biu.ac.il}

\author{Luie Polev}
\address{Department of Mathematics,
	Bar-Ilan University, 52900 Ramat-Gan, Israel}
\ead{luiepolev@gmail.com}

\begin{abstract}
A Hausdorff topological group is called minimal if it does not admit a strictly coarser Hausdorff group topology.
   This paper mostly deals with the topological group $\H_+(X)$ of order-preserving homeomorphisms of a compact 
   linearly ordered connected space $X$.
    We provide a sufficient condition on $X$ under which the topological group $\H_+(X)$ is minimal. This condition is satisfied, for example, by:
   the unit interval, the ordered square, the extended long line and the circle (endowed with its cyclic order).
   In fact, these groups are even $a$-minimal, meaning, in this setting, that the compact-open
   topology on $G$ is the smallest Hausdorff group topology on $G$. 
   One of the key ideas is to verify that
   for such $X$ the Zariski and the Markov topologies on the group $\H_+(X)$ coincide
   with the compact-open topology. 
   The technique in this article is mainly based on a work of Gartside and Glyn \cite{Gartside}.
\end{abstract}

\begin{keyword}
$a$-minimal group,
Markov's topology, minimal groups, compact LOTS, order-preserving homeomorphisms, Zariski's topology.

\MSC[2010] 57S05, 54F05, 54H15.

\tnotetext[1]{This research was supported by a grant of Israel Science Foundation (ISF 668/13).}

\end{keyword}

\end{frontmatter}



\section{Introduction}

A Hausdorff topological group $G$ is \textit{minimal} (\cite{Doitch}, \cite{Steph}) if it does not admit a strictly coarser Hausdorff group topology or, equivalently, if every injective continuous group homomorphism $G \to P$ into a Hausdorff topological group is a topological group embedding.

All topological spaces are assumed to be Hausdorff and completely regular 
 (unless stated otherwise).
Let $X$ be a compact topological space. Denote by $\H(X)$ the group of all homeomorphisms of $X$, endowed with the compact-open topology $\tau_{co}$. In this setting $\H(X)$ is a topological group and the natural action $\H(X)\times X \to X$  is continuous.

Clearly, every compact topological group is minimal.
The groups $\mathbb R$ and $\mathbb Z$, on the other hand, are not minimal. Moreover, Stephenson showed in \cite{Steph} that an LCA group is minimal if and only if it is compact.
 Nontrivial examples of minimal groups include $\mathbb{Q} / \mathbb{Z}$ 
 with the quotient topology, 
 \cite{Steph}, and $S(X)$, the symmetric group of an infinite set (with the pointwise topology). The minimality of the latter was proved by Gaughan \cite{Gaughan} 
 and (independently) by Dierolf and Schwanengel \cite{DS}. 
For more information on minimal groups we refer to the surveys \cite{CHR}, \cite{Dikr},
\cite{MegDik} and the book \cite{DPS}.


\vskip 0.2cm

The following is a question of Stoyanov (cited in \cite{Arhan}, for example): 
\begin{question} \label{q:St} (Stoyanov)
	Is it true that for every compact homogeneous space $X$ the topological group $\H(X)$ is minimal?
\end{question}
One important positive example of such a space is the Cantor cube $2^{\omega}$. Indeed, in \cite{Gam} Gamarnik proved that $\H(2^{\omega})$ is minimal. Recently van Mill (\cite{vanMill2012}) provided a counterexample to Question \ref{q:St}
 proving that for the $n$-dimensional Menger universal continuum $X$, where $n>0$, the group $\H(X)$ is not minimal.

It is well known that the Hilbert cube $[0,1]^{\omega}$ is a homogeneous compact space as well. The following question of Uspenskij \cite{Usp2008} remains unanswered:
is the group $\H ([0,1]^{\omega})$ minimal?

\begin{defi} \ \label{d:a}
\begin{enumerate}
	\item \cite{MegDik}
	A topological group $G$ is \textit{$a$-minimal} if its topology is the smallest possible Hausdorff group topology on $G$.
\item \cite{MegDik} A compact space $X$ is \textit{$M$-compact} (\textit{$aM$-compact}) if the topological group $\H(X)$ is minimal (respectively, $a$-minimal).
\item A compact ordered space $X$ is \textit{$M_+$-compact} (\textit{$aM_+$-compact}) if the topological group $\H_+(X)$ of all order-preserving homeomorphisms of $X$ is minimal (respectively, $a$-minimal).
\end{enumerate}
\end{defi}

Several questions naturally arise at this point:

\begin{question} \label{q:M} \
\begin{enumerate}
\item \cite{MegDik} Which (notable) compact spaces are $M$-compact? $aM$-compact?
\item Which compact ordered spaces are $M_+$-compact? $aM_+$-compact?
\end{enumerate}
\end{question}

The two point compactification of $\Z$ is a compact LOTS $X$ such that $\H_+(X)$ 
and $\H(X)$ are not minimal (Example \ref{notM_+}). Thus not every compact LOTS is $M_+$-compact or $M$-compact.

Clearly, every $a$-minimal group is minimal.
It is well known that $(\Z, \tau_p)$ with its $p$-adic topology is a minimal topological group.
Since such topologies are incomparable for different $p$'s,
it follows that $(\Z, \tau_p)$ is not $a$-minimal.

 Recall a few results:
	\begin{enumerate} 
\item  (Gaughan \cite{Gaughan}) The symmetric group $S(X)$ is $a$-minimal. Since $\H(X^*)$ is precisely $S(X)$ we obtain that the $1$-point compactification $X^*$ of a discrete set $X$ is $aM$-compact. 	
\item (Banakh-Guran-Protasov \cite{BGP})
Every subgroup of $S(X)$ that contains
$S_{\omega}(X)$ (permutations of finite support) is $a$-minimal (answers a question of Dikranjan \cite{Lukach}).
\item (Gamarnik \cite{Gam}) $[0,1]^n$ is $M$-compact 
(for $n \in \N$) 
if and only if $n=1$.
\item (Gartside and Glyn \cite{Gartside})
$[0,1]$ and $\mathbb{S}^1$ are $aM$-compact.
\item (Gamarnik \cite{Gam}) The Cantor cube $2^{\omega}$ is $M$-compact.
\item (Uspenskij \cite{Usp2001}) Every $h$-homogeneous compact space is $M$-compact.
\item (van Mill \cite{vanMill2012}) 
$n$-dimensional Menger universal continuum $X$, where $n>0$, is not $M$-compact (answers Stoyanov's Question \ref{q:St}).
\end{enumerate}

Recall that
a zero-dimensional compact space $X$ is \emph{$h$-homogeneous} if all non-empty clopen subsets of $X$ are homeomorphic to $X$. In particular, $2^{\omega}$ is $h$-homogeneous. Hence, (6) is a generalization of (5).

The concept of an $a$-minimal group
is in fact an intrinsic algebraic property of an abstract group $G$ (underlying a given topological group).
$a$-minimality is interesting for several reasons. For instance, it is strongly related to some fundamental topics like Markov's and Zariski's topologies.

 For additional information about $a$-minimality (and minimality) see the recent survey \cite{MegDik}. For Markov's and Zariski's topologies see \cite{DikSh}, \cite{DikSh10}, \cite{BGP}, \cite{DikTol}. We recall 
 the definitions.

\begin{defi} \label{d:ZarMar}   
Let $G$ be a group.
\begin{enumerate}
	\item  The \textit{Zariski topology} $\mathfrak{Z}_G$ is generated by the sub-base consisting of the sets $\{x\in G: x^{\eps_1}g_1x^{\eps_2}g_2\cdots x^{\eps_n}g_n\neq e\}$, where $e$ is the unit element of $G$, $n\in \mathbb N$, $g_1,...,g_n\in G$, and $\eps_1,...,\eps_n\in \{-1,1\}$.
\item The \textit{Markov topology} $\mathfrak{M}_G$ is the infimum (taken in the lattice of all topologies on $G$) of all \emph{Hausdorff} group topologies on $G$. 
\end{enumerate}
\end{defi}

Note that $(G,\mathfrak{Z}_G)$ and $(G,\mathfrak{M}_G)$ are quasi-topological groups. That is the inverse and the translations are continuous. They are not necessarily topological groups. In fact, if $G$ is abelian then $\mathfrak Z_G$ and $\mathfrak M_G$ are not group topologies, unless  
$G$ is finite, \cite[Corollary 3.6]{DikSh10}.
Here we give some simple properties. Regarding assertion (3) in the following lemma see for example \cite[Defition 2.1]{MegDik}.

\begin{lemma} \label{l:ZarMark}
	Let $G$ be an abstract group. Suppose that $\tau$ is a Hausdorff group topology on $G$. Then
	\begin{enumerate}[(1)]
		\item $\mathfrak Z_G \seq \mathfrak M_G \seq \tau$.
		\item  $\mathfrak Z_G = \mathfrak M_G = \tau$  if and only if $\tau \seq \mathfrak Z_G$. 
		In this case $(G,\tau)$ is $a$-minimal.
		\item $\mathfrak M_G$ is a (not necessarily, Hausdorff) group topology if and only if $(G,\mathfrak M_G)$ is an $a$-minimal topological group.
	\end{enumerate}
\end{lemma}
\begin{proof}
\begin{enumerate}[(1)]
    \item Follows directly from the definitions. 
    \item Follows from (1).
    \item Note that $\mathfrak M_G$ is always a $T_1$-topology. Hence, if $\mathfrak M_G$ is a topological group topology then it is Hausdorff. Taking into account the definition of $\mathfrak M_G$ we can conclude that this topology is the smallest Hausdorff group topology on $G$. Hence, $(G,\mathfrak M_G)$ is $a$-minimal.
\end{enumerate}
\end{proof}

\begin{question} \label{q:Markov} [Markov]
For what groups $G$ the Markov and Zariski topologies coincide?
\end{question}

A review of some old and new partial answers can be found in \cite{DikTol}. 
Below, in Theorem \ref{t:new}, we give additional examples of groups
for which $\mathfrak{Z}_G=\mathfrak M_G$.

In the present paper we mainly deal with the groups $\H_+(X)$.
Given an ordered compact space $X$, we are interested in the group $\H_+(X)$ of order-preserving homeomorphisms.
For a compact space $X$ the group $\H(X)$ is complete (with respect to the  two-sided uniformity) and therefore $\H_+(X)$ is also complete (as a closed subgroup of a complete group).

In certain cases the minimality of $\H(X)$ can be deduced from the minimality of $\H_+(X)$, as the following  lemma shows.

\begin{lemma} \label{link}
Let $X$ be a compact LOTS such that $\H_+(X)$ is minimal. If $\H_+(X)$ is a co-compact subgroup of $\H(X)$, then $\H(X)$ is minimal.
\end{lemma}

This lemma is a corollary of Lemma \ref{general_lemma}. 
\emph{Co-compactness} 
of $\H_+(X)$ in $\H(X)$ means that the coset space $\H(X) / \H_+(X)$ is compact.

If $X$ is a linearly ordered continuum, then by Lemma \ref{l:2} the subgroup $\H_+(X)$ has at most index $2$
in $\H(X)$. 
So, in this case, from the minimality of $\H_+(X)$ we can deduce by Lemma \ref{link} the minimality of $\H(X)$.
For example, it is true for $X=[0,1]$.
Note that $\H[0,1]=\H_+[0,1] \leftthreetimes \mathbb Z_2$, the topological semidirect product of $\H_+[0,1]$ and $\mathbb Z_2$, where $\Z_2$ is the two element group. However, in general, it is unclear how to infer the minimality of a topological group $G$ from the minimality of $G \leftthreetimes \mathbb Z_2$. For instance, in \cite[Example $4.7$]{MegDik} it is shown that there exists a non-minimal group $G$ such that $G\leftthreetimes \mathbb Z_2$ is minimal.

\vskip 0.3cm
Recall the following result of Gartside and Glyn:

\begin{thm}\textnormal{\cite{Gartside}} \label{GG}
	For any metric one dimensional manifold (with or without boundary) $M$, the
	compact-open topology on the full homeomorphism group $\H(M)$ is the unique minimum 
	Hausdorff group topology on $\H(M)$.
\end{thm}

The one dimensional compact manifolds, up to homeomorphism, are the closed interval $[0,1]$ and the circle $\mathbb{S}^1$. In view of Definition \ref{d:a} this result implies the following. 

\begin{thm} \label{GGcor} \textnormal{\cite{Gartside}}
	$\H[0,1]$ and $\H(\mathbb{S}^1)$ are $a$-minimal groups.
\end{thm}

Extending some ideas of Gartside-Glyn \cite{Gartside} to linearly ordered spaces 
we give some new results about minimality of the groups $\H_+(X)$ of order preserving homeomorphisms. 

\begin{thm*}[see Theorem \ref{t:new}]
	Let $(X,\tau_{\leq})$ be a compact connected LOTS that satisfies the following condition:
	\begin{enumerate}[(A)]
	\item
	for every pair of elements $a<b$ in $X$ the group $\H_+[a,b]$ is nontrivial.
	\end{enumerate}
	Then:
	\begin{enumerate}[(1)] 
		\item For the topological group  $G=\H_+(X)$ and $G=\H(X)$ the Zariski and Markov topologies 
		coincide with the compact-open topology.
		That is, $\mathfrak{Z}_G=\mathfrak M_G=\copt$.
		\item The topological
		groups $\H_+(X)$ and $\H(X)$ are $a$-minimal.
		\item  $X$ is $aM_+$-compact and $aM$-compact.
	\end{enumerate}
\end{thm*}

According to results of Hart and van Mill \cite{Hvan} (see Section \ref{s:rigid})
there exists a connected compact LOTS $X$ which is $\H_+$-rigid, that is, $\H_+(X)$ is trivial 
(in fact, $\H(X)$ is trivial).  
Hence, condition (A) of the theorem above is not always satisfied for general ordered continua.
Moreover, one may derive from results of \cite{Hvan} that there exists a connected compact LOTS $X$
for which $\H_+(X)=\H(X)=\Z,$ a discrete copy of the integers $\Z$, 
 and $\H[c,d]$ is trivial for some pair $c<d$ in $X$ (Proposition \ref{p:HartMill}).

In Section \ref{s:ex} we give some concrete examples of spaces that satisfy 
condition (A) of Theorem \ref{t:new}.
The following linearly ordered spaces $X$ are $aM_+$-compact, that is the groups $\H_+(X)$ 
are $a$-minimal:
\begin{enumerate}
	\item $[0,1]$;
    \item the lexicographically ordered square $\mathcal I^2$;
	\item the extended long line $\mathcal{L}^*$;
	\item the ordinal space $[0,\kappa]$;
	\item the unit circle $\mathbb{S}^1$
	(in this case we work with a \emph{cyclic order}, Definition \ref{d:cycord}).
\end{enumerate}

Note that the groups $\H_+(X)$ play a major role in many research lines.
See, for example, \cite{Ghys, Pestov2006, GM-tame}.

\vskip 0.3cm

\noindent {\bf Acknowledgments.} We thank 
R. Ben-Ari, D. Dikranjan, K.P. Hart, J. van Mill and M. Shlossberg for valuable suggestions.
We also thank the referee for constructive criticism and many improvements.

\section{Preliminaries}

In what follows, every compact topological space will be considered as a uniform space with respect to its natural (unique) uniformity.

For a topological group $(G,\gamma)$ and its subgroup $H$ denote by $\gamma / H$ the natural quotient topology on the coset space $G/H$.

\begin{lemma} \textnormal{(Merson's Lemma)} \label{merson}
	Let $(G, \gamma)$ be a not necessarily Hausdorff topological group and $H$ be a not necessarily closed subgroup of $G$. If $\gamma_1 \seq \gamma$ is a coarser group topology on $G$ such that $\gamma_1|_H=\gamma|_H$ and $\gamma_1/H=\gamma/H$, then $\gamma_1=\gamma$.
\end{lemma}

\begin{lemma}\label{general_lemma}
	Let $H$ be a co-compact complete subgroup of a topological group $G$. If $H$ is minimal then $G$ is minimal too.
\end{lemma}
\begin{proof}
	Denote by $\tau$ the given topology on $G$, and let $\gamma \subseteq \tau$ be a coarser Hausdorff group topology. Since $H$ is minimal, we know that $\gamma |_H=\tau |_H$. Furthermore, $H$ is $\gamma$-closed in $G$ because $H$ is complete. Since $(G/H, \gamma/H)$ is Hausdorff  and  $(G/H, \tau/H)$ is compact we have $\gamma /H=\tau /H$. Thus, by Merson's Lemma \ref{merson}, we conclude that $\gamma=\tau$.
\end{proof}

\subsection{Ordered topological spaces}

 A \emph{linear order} on a set $X$ is, as usual, a binary relation $\leq$ which is reflexive, antisymmetric, transitive
and satisfies in addition the totality axiom: for all $a,b\in X$ either $a\leq b$ or $b\leq a$.

For a set $X$ equipped with a linear order $\leq$, the \textit{order topology} (or \textit{interval topology})  $\tau_{\leq}$ on $X$ is generated by the subbase that consists of the intervals $(\leftarrow, a)=\{x\in X: x<a \}$, \space $(b,\rightarrow)=\{x\in X: b<x \}$.
A \textit{linearly ordered topological space} (or LOTS) is a triple $(X,\tau_{\leq}, \leq)$ where $\leq$ is a linear order on $X$ and $\tau_{\leq}$ is the order topology on $X$.
For every pair $a<b$ in $X$ the definition of the intervals $(a,b), [a,b]$
is understood. Every linearly ordered compact space $X$ has the smallest and the greatest element; 
so, $X=[s,t]$ for some $s,t \in X$.

Sometimes we say: \emph{linearly ordered continuum}, instead of \emph{compact and connected LOTS}.

\begin{lemma} \label{l:2}
	Let $(X,\tau_{\leq})$ be a linearly ordered continuum.
	Then every $f\in \H(X)$ is either order-preserving or order-reversing. 
	In particular, the index of $\H_+(X)$ in $\H(X)$ is at most 2.
\end{lemma}
\begin{proof}
	Assume for contradiction that there exists $f\in \H(X)$ such that $f$ is neither order-preserving nor order-reversing. Thus there exist three points $x_1,x_2,x_3 \in X$ such that $x_1< x_2< x_3$ and  either $f(x_1) < f(x_2) \wedge f(x_2)> f(x_3)$ or $f(x_1) > f(x_2) \wedge f(x_2)< f(x_3)$.  Both cases lead to a contradiction. We give the details for the first case
	(the second case is similar).
	
	Suppose  $x_1 < x_2 < x_3$ and $f(x_1) < f(x_2) \wedge f(x_2) > f(x_3)$. Since $X$ is linearly ordered
	there are two possibilities to consider.
	\begin{enumerate}
		\item $f(x_1)< f(x_3)< f(x_2)$: then by the Intermediate Value Theorem (applied to the interval $[x_1,x_2]$) there exists $x_1 < x_0< x_2$ such that $f(x_0)=f(x_3)$, which is a contradiction since $f$ is $1-1$.
		\item $f(x_3)< f(x_1)< f(x_2)$: then again by the Intermediate Value Theorem (applied to the interval $[x_2,x_3]$) there exists $x_2< x_0< x_3$ such that $f(x_0) = f(x_1)$, which is a contradiction because $f$ is $1-1$.
	\end{enumerate}
	Each case leads to a contradiction, and this fact concludes the proof.
\end{proof}

In the sequel we use several times the following simple "localization lemma".
\begin{lemma} \label{r:ext}
Let $X$ be a LOTS and let $a < b$ be a given pair of elements in $X$.
	If $h \in \H_+[a,b]$, then for
	the natural extension $\hat{h}\co X \to X$, with $\hat{h}(x)=x$ for every $x \in X \setminus (a,b) =  (\leftarrow, a] \cup [b,\rightarrow)$, we have
	$\hat{h} \in \H_+(X)$.
\end{lemma}

The idea of the following lemma was kindly provided to us by K.P. Hart.

\begin{lemma} \label{l:Hart}
	Let $X$ be a linearly ordered continuum. The following conditions are equivalent:
	\begin{enumerate}
		\item [(A)] for every pair of elements $a<b$ in $X$ the group $\H_+[a,b]$ is nontrivial,
		\item [(B)] for every pair of elements $a<b$ in $X$ the group $\H_+[a,b]$ is nonabelian.
	\end{enumerate}
\end{lemma}
\begin{proof} Let $a <b$ in $X$. Assuming (A)
	there exists a nontrivial $h_1 \in \H_+[a,b]$. 
	So, $h_1(u) \neq u$ for some $u \in (a,b)$.  We can suppose that $a < u< h_1(u) <b $
	(indeed, if $h_1(u)<u$, replace $h_1$ by $h_1^{-1}$ and $u$ by $h_1(u)$).
	Since $X$ is a continuum, the interval $(u,h_1(u))$ is nonempty. Choose an arbitrary $v \in (u,h_1(u))$.
	By the continuity of $h_1$ there exists a sufficiently small neighbourhood $O$ of $u$ such that
	$$s < v < h_1(t)$$
	for every $s,t \in O$. Without restriction of generality we can assume that $O$ is the interval $[x_1,x_2]$, where $x_1<x_2$.
	Clearly, $h_1(x_1) < h_1(x_2)$, so $$a<x_1< x_2< h_1(x_1) < h_1(x_2)<b.$$
	Now apply condition (A) to the interval $[h_1(x_1),h_1(x_2)]$. There exists  a nontrivial
	$h_2 \in \H_+[h_1(x_1),h_1(x_2)]$.
	Similarly, as for $h_1$ and $[a,b]$, one may choose, for $h_2$ and $[h_1(x_1),h_1(x_2)]$,
	a subinterval $[y_1,y_2]$ of $[h_1(x_1),h_1(x_2)]$ such that
	$$h_1(x_1) < y_1<y_2< h_2(y_1) < h_2(y_2)<h_1(x_2).$$
	We can treat $h_2$ as an element of $\H_+[a,b]$ by the natural extension
	(assuming that $h_2(x)=x$ outside of $[h_1(x_1),h_1(x_2)]$).
	
	The interval $[h_1^{-1}(y_1),h_1^{-1}(y_2)]$ is a nonempty subinterval of $[x_1,x_2]$.
	Now observe that for every $z \in [h_1^{-1}(y_1),h_1^{-1}(y_2)]$ we have 
	 $z < h_1(x)$. Therefore, $h_2(z)=z$. So, we get
	$$h_1(h_2(z))=h_1(z) \in [y_1,y_2],$$ while $$h_2(h_1(z)) \in [h_2(y_1) ,h_2(y_2)].$$
	Since $y_2 < h_2(y_1)$, we can
	conclude that $h_2 \circ h_1 \neq h_1 \circ h_2$ and $\H_+[a,b]$ is nonabelian.
\end{proof}

\begin{defi}(see, for example, \cite{Cech,Kok}) \label{d:cycord}
	A ternary relation $R \subseteq X^3$ on a set $X$ is said to be a \textit{cyclic ordering} if:
	\begin{enumerate}
		\item
		$
		\begin{cases}
		a\neq b \neq c \neq a\\
		(a,b,c) \notin R
		\end{cases}
		$
		$\Leftrightarrow (c,b,a) \in R$.
		\item $(a,b,c) \in R \Rightarrow (b,c,a) \in R$.
		\item $
		\begin{cases}
		(a,b,c) \in R \\
		(a,c,d) \in R
		\end{cases}
		$
		$\Rightarrow (a,b,d) \in R$.
	\end{enumerate}
\end{defi}

Let $X$ be a topological space and $R$ be a cyclic ordering on $X$. A
homeomorphism $f\co X \to X$ is  \emph{orientation preserving} 
if $f$ preserves $R$, meaning that
$(z,y,x) \in R$ implies $(f(z),f(y),f(x)) \in R$. The set of all such autohomeomorphisms is a 
subgroup of $\H(X)$ which we denote by $\H_+(X)$.

\section{Order-preserving homeomorphisms and \texorpdfstring{$a$}{}-minimality}

Using some results of Nachbin
we extend the ideas of Gartside and Glyn \cite{Gartside} to compact connected linearly ordered spaces (Theorem \ref{t:new}).

For the purposes of this section we fix the following notations.
Let $(X,\tau_{\leq})$ be a compact LOTS with its unique compatible uniform structure $\mu$ and denote $s=\min X, t= \max X$.
For every $f \in C(X)$ and $\eps >0$ define
$$
U_{f,\eps}:=\{(x,y) \in X \times X: \ |f(x)-f(y)| \leq \eps\}.
$$
Denote by $C_+(X,[0,1])$ the set of all continuous order-preserving maps $f\co X \to [0,1]$.
	
	\begin{lemma} \label{l:Nachbin} \emph{(Nachbin \cite{Nachbin})}
		Let $X$ be a compact LOTS.
		\begin{enumerate}
			\item $C_+(X,[0,1])$ separates the points of $X$.
			\item The family $\{U_{f,\eps}: f \in C_+(X,[0,1]), \eps >0\}$ is a subbase of the uniformity $\mu$ for every compact LOTS $X$.
		\end{enumerate}
	\end{lemma}
	\begin{proof} (1) It is a fundamental result of Nachbin \cite[p. 48 and 113]{Nachbin}.
		
		(2)  Use (1) and the following observation. For every compact space $X$ and a point-separating family $F$ of (uniformly) continuous functions $X \to [0,1]$, the corresponding weak uniformity $\mu_F$ on $X$ is just the natural unique compatible uniformity $\mu$ on $X$.
		The family  of entourages $\{U_{f,\eps}: f \in F, \eps > 0\}$ is a uniform subbase of $\mu=\mu_F$.
	\end{proof}
	
\begin{defi} \label{d:connect}
	Let $\alpha \in \mu$ be an entourage. We say that a finite chain
	$A:=\{c_0, c_1, \cdots, c_n\}$ in $X$ is an \textit{$\alpha$-connected net} if :
	\begin{enumerate}
		\item $s=c_0 \leq c_1 \leq \cdots \leq c_n=t$;
		\item  $(x,y) \in \alpha$ for every $x,y \in [c_i,c_{i+1}]$ 
		and $0 \leq  i \leq n-1$.
	\end{enumerate}
	Notation: $A \in \Gamma(\alpha)$.
\end{defi}

Note that $(x,y) \in \alpha^2$ for every $x \in [c_k,c_{k+1}]$ and $y \in [c_{k+1}, c_{k+2}]$.

\begin{lemma} \label{l:connected}
	Let $(X,\tau_{\leq})$ be a compact LOTS with its unique compatible uniform structure $\mu$. The following are equivalent:
	\begin{enumerate}
		\item $X$ is connected;
		\item for every $\alpha \in \mu$ there exists an $\alpha$-connected net.
	\end{enumerate}
\end{lemma}
\begin{proof} (1) $\Rightarrow$ (2)
	
	In the setting of Definition \ref{d:connect} every finite chain which contains an $\alpha$-connected net is also an $\alpha$-connected net.
	It follows that it is enough to verify the definition for entourages from any given uniform subbase of $\mu$.
	So, in our case, by Lemma \ref{l:Nachbin}, it is enough to check that there exists an $\alpha$-connected net for every $\alpha = U_{f,\eps}$.
	We have to show that $\Gamma({U_{f,\eps}})$ is nonempty for every $f \in C_+(X,[0,1])$ and every $\eps >0$.
	
	Since $X$ is connected and compact the continuous image $f(X) \subseteq [0,1]$ is a closed subinterval, say $f(X)=[u,v]$.
	
	Fix $n \in \N$ large enough such that $\frac{v-u}{n} \leq \eps$. For every natural $i$ with $0<i<n$ choose $c_i \in X$ with $f(c_i)= \frac{(v-u)i}{n}+u$ and $c_0=s$, $c_n=t$.
	Then $$A:=\{c_0,c_1, \dots, c_n\} \in \Gamma(U_{f,\eps}).$$ 
	Indeed, since $f$ is order-preserving, for every $x,y \in X$ with $x,y \in [c_i,c_{i+1}]$ we have
	$$f(x), f(y) \in [f(c_i),f(c_{i+1})].$$
	So $|f(x)-f(y)| \leq \frac{v-u}{n} \leq \eps$. Therefore, $(x,y) \in \alpha=U_{f,\eps}$.

	(2) $\Rightarrow$ (1)
	
	Assume to the contrary that $X$ is not connected. Since $X$ is a compact LOTS it follows that the order is not dense. That is, there exist $a <b$ in $X$ such that the interval $(a,b)$ is empty. Then the function $f\co X \to [0,1]$, where $f(x)=0$ for $x \leq a$ and $f(x)=1$ for $b \leq x$ is continuous.
	Choose any $0<\eps <1$ and define
	$\alpha:= U_{f,\eps} \in \mu$. Then $\Gamma({\alpha})$ is empty.
\end{proof}

Assertion $(2)$ of the following theorem for $X:=[0,1]$ generalizes a result of \cite{Gartside} mentioned above in Theorem \ref{GGcor}.
	We modify the arguments of \cite{Gartside} and use
	Lemmas \ref{l:ZarMark}, \ref{l:Hart} and \ref{l:connected}.
	
\begin{thm} \label{t:new}
	Let $(X,\tau_{\leq})$ be a compact connected LOTS that satisfies the following condition:
	\begin{enumerate}[(A)]
\item
for every pair of elements $a<b$ in $X$ the group $\H_+[a,b]$ is nontrivial.
\end{enumerate}
	Then:
	\begin{enumerate}[(1)]
		\item For the topological groups $G=\H_+(X)$ and $G=\H(X)$ the Zariski and Markov topologies 
		coincide with the compact-open topology.
		That is, $\mathfrak{Z}_G=\mathfrak M_G=\copt$.
		\item The topological
		groups $\H_+(X)$ and $\H(X)$ are $a$-minimal.
		\item  $X$ is $aM_+$-compact and $aM$-compact.
	\end{enumerate}
\end{thm} 
\begin{proof} Assertion $(2)$  follows from $(1)$ by applying Lemma \ref{l:ZarMark}. By Definition \ref{d:a}
	 assertion $(3)$ is a reformulation of $(2)$. So it is enough to prove $(1)$.  	
	
Below $G$ denotes one of the groups $\H_+(X)$ or $\H(X)$.
Denote by $\tau_{co}$ the (compact-open) topology on $G$.
By Lemma \ref{l:ZarMark} it is equivalent to show that $\copt \seq \mathfrak Z_G$.

	For every interval $(a,b) \seq X$ (with $a<b$) the group $\H_+[a,b]$ is
	nontrivial (condition (A)) and thus, by Lemma \ref{l:Hart}, this group is nonabelian. Taking into account Lemmas \ref{l:Hart} and \ref{r:ext} choose $p, q \in \H_+(X)$ such that $pq \neq qp$ and $p(x)=q(x)=x$ for every $x \notin (a,b)$.
	Define
	\begin{equation} \label{T}
		T(a,b):= \{g \in G: gpg^{-1} \ \text{does not commute with} \ q\}.
	\end{equation}

\begin{claim}\label{TabOpen}	
	$e \in T(a,b) \in \mathfrak Z_G$. 
\end{claim}
\begin{proof}
	Indeed, rewrite the definition of $T(a,b)$ to obtain
		$$	T(a,b)=  \{g \in G: (gpg^{-1})q(gpg^{-1})^{-1}q^{-1} \neq e \}$$
		and use Definition \ref{d:ZarMar} to conclude that $T(a,b) \in \mathfrak Z_G$.
		The fact that $e \in T(a,b)$ is trivial by the choice of $p,q$.
	\end{proof}
	
\begin{claim} \label{Cl:T}
For every $g \in T (a,b)$ there exists $x \in (a, b)$ such that $g(x) \in (a, b).$ That is,
	$g(a,b) \cap (a,b) \neq \emptyset$ \ $\forall g \in T(a,b)$.
\end{claim}	
\begin{proof} 
	Assuming the contrary, there exists $g \in T(a,b)$ such that
	$g(a,b) \cap (a,b) = \emptyset$. Equivalently, $(a,b) \cap g^{-1}(a,b) = \emptyset$. Hence, $g^{-1}(x) \notin (a,b)$ for every $x \in (a,b)$.
	By the choice of $p$ we have $pg^{-1}(x)=g^{-1}(x)$ and so $gpg^{-1}(x)=x$ for every $x \in (a,b)$. On the other hand, $q(x)=x$
	for every $x \in X \setminus (a,b)$ (by the choice of $q$). It follows that  $gpg^{-1}$ and $q$ commute, which contradicts the definition of $T(a,b)$ in (\ref{T}).
\end{proof}

	Let $\alpha$ be the collection of all finite intersections of $T (a,b)$'s. 
	By Claim \ref{TabOpen} (using that $\mathfrak Z_G$ is a topology) 
	we obtain $\alpha \subseteq \mathfrak Z_G$. 
	Both $\copt$ and $\mathfrak Z_G$ are completely determined by the neighbourhood base at $e \in G$.
	 So, in order to see that $\copt \seq \mathfrak{Z}_G$ it suffices to show
	the following.

\begin{claim}\label{co_in_Zariski}
 Every open neighbourhood $U$ of $e$ in $G$, with the compact-open topology $\copt$, contains an element $T$ from $\alpha$. 
\end{claim}	
\begin{proof}
	Let $\mu$ be the unique compatible uniformity on $X$.
	A basic neighbourhood of $e$ has the form:
	$$
	O_{\eps}:=\{g \in G: (g(x),x) \in \eps \ \ \ \forall \ x \in X\},
	$$
	where $\eps \in \mu$. Choose a symmetric entourage $\eps_1 \in \mu$ such that $\eps_1^2 \subseteq \eps$.
	For $\eps_1$ by Lemma \ref{l:connected} choose an $\eps_1$-connected net
	$$c_0 < c_1 < \cdots < c_n$$
	of $X$. We can suppose that $X$ is nontrivial and $n>0$. 
	
	By Equation \ref{T}, we have the corresponding $T(c_i,c_{i+1}) \subseteq G$ for every index $0 \leq i \leq n-1$. Define
	$$
	T:= \bigcap_{i=0}^{n-1} T(c_i,c_{i+1}).
	$$ 
Now it is enough to show:
	
	\vskip 0.2cm

	\begin{equation} \label{claim4}
T \subseteq O_{\eps}.
	\end{equation}
	
	\vskip 0.3cm
	
	Assuming the contrary let $h \in T$ but $h \notin O_{\eps}$.
	Then there exists $x \in X$ such that $(h(x),x) \notin \eps$.
	Pick minimal index $k$ between $0$ and $n-1$ such that $x \in [c_k,c_{k+1}]$. Then by a remark after  Definition \ref{d:connect}
	 we have $(x,y) \in \eps_1^2 \subseteq \eps$ for every $y \in [c_{k-1}, c_{k+2}]$. If $k=0$, we replace $c_{k-1}$ by $c_0$. Similarly, we replace $c_{k+2}$ by $c_n$ if $k=n-1$.
	
	  Hence,
	\begin{equation} \label{g1}
	h(x) \in X \setminus [c_{k-1}, c_{k+2}]= [c_0,c_{k-1}) \cup (c_{k+2},c_n].
	\end{equation}
	Note that one of the intervals in the union can be empty.

	\vskip 0.2cm
	
	From Claim \ref{Cl:T} for every index $0 \leq i \leq n-1$ choose $x_i$ such that
	\begin{equation} \label{g4}
	x_i, h(x_i) \in (c_i, c_{i+1}).
	\end{equation}

	We show that there is no such $h \in G$.
	By Lemma \ref{l:2} any autohomeomorphism $h \in \H(X)$ is either order-preserving or
	order-reversing. By Equation \ref{g4} we have $h(x_i) < h(x_{i+1})$, where $x_i < x_{i+1}$.
	So, $h$ can be only order-preserving.
	
	Now, we show that $h$ is not order-preserving. Indeed,
	we have the following two cases:
	
	\begin{enumerate}
		\item [(1)] $h(x) \in (c_{k+2},c_n]$.

		Then, $h(x_{k+1}) < h(x)$, while $x < x_{k+1}$.

		\item [(2)] $h(x) \in [c_0,c_{k-1})$.

		Then, $h(x) < h(x_{k-1})$, while $x_{k-1} < x$. 	 	
	\end{enumerate}
	
	In both cases we get a contradiction.
	This completes the proof of Equation \ref{claim4} and hence of our theorem.
\end{proof}
\end{proof}

\begin{corol} \label{c:suff}
	Let $(X,\tau_{\leq})$ be a compact connected LOTS that satisfies the following condition:
	\begin{enumerate}[(C)]
		\item for every pair of elements $a<b$ in $X$ there exist $c,d \in X$ with $a\leq c <d \leq b$ such that
		$[c,d]$ is separable (equivalently, the subspace $[c,d] \subseteq X$ is homeomorphic to the real unit interval $[0,1]$).
	\end{enumerate} 
	Then $\mathfrak{Z}_G=\mathfrak M_G=\copt$ and the groups $G=\H_+(X)$, $G=\H(X)$ are $a$-minimal (that is, $X$ is $aM_+$-compact and $aM$-compact). 	 	
\end{corol}
\begin{proof}
	Recall (see, for example, \cite[Exercise 6.3.2]{Engel}) that a separable linearly ordered continuum is homeomorphic to $[0,1]$. 
	It is well known and easy to see that
	 $\H_+[0,1]$ is nonabelian (Section \ref{s:int}).
	Also, up to the inversion, there exists only one linear order on $[0,1]$ inducing the natural topology \cite[Cor. 4.1]{Kok}. We see that (C) implies that $\H_+[c,d]$ (being a copy of $\H_+[0,1]$) is nonabelian.
	So, we can apply Theorem \ref{t:new}.
\end{proof}

\section{Some examples} \label{s:ex}

\subsection{Not every compact LOTS is \texorpdfstring{$M_+$}{}-compact}

The following example shows that $\H_+(X)$ is not necessarily minimal.

\begin{example} \label{notM_+}
	Denote by $\mathbb{Z}^*$ the two-point compactification of $\mathbb{Z}$. One can easily verify that $\H_+(\mathbb{Z}^*)$ is a discrete copy of $\mathbb{Z}$ and thus not minimal. That is, the compact LOTS $\mathbb{Z}^*$ is not $M_+$-compact. 
	Note that $\mathbb{Z}^*$ is also not $M$-compact as it directly follows from \cite[Theorem 4.25]{MegDik}.
\end{example}

\vskip 0.3cm

\subsection{Rigid ordered compact spaces}
\label{s:rigid}

Let us say that a topological space $X$ is \emph{$\H$-rigid} if the group $\H(X)$ is trivial.
Similarly, let us say that a linearly ordered space $X$ is \emph{$\H_+$-rigid} if the group $\H_+(X)$ is trivial.
Certainly, if $X$ is $\H$-rigid then it is also $\H_+$-rigid.
There are many known examples of $\H$-rigid compact spaces, and in particular of compact ordered  $\H$-rigid spaces.
Most of the examples of the latter kind (Jonsson, Rieger, de Groot-Maurice) are zero-dimensional. 
It seems that the first ("naive") example of a nontrivial \emph{connected} compact ordered $\H$-rigid space was constructed by Hart and van Mill \cite{Hvan}. Note also that, under the \emph{diamond principle}, there exists an
$\H$-rigid Suslin continuum 
(Jensen, see in \cite[p. 268]{Tod}).

Using results of \cite{Hvan}, one may show the following.

\begin{prop} \label{p:HartMill}
There exists an ordered continuum $X$ with $\H_+(X) =\H(X) = \Z$, 
a discrete copy of the integers. 
\end{prop}

So, we get
a \emph{connected} compact LOTS $X$ such that $\H_+(X)$ is not minimal (or, $X$ is not $M_+$-compact). 
Hence, Theorem \ref{t:new} does not remain true for general ordered continua.

We sketch the proof of Proposition \ref{p:HartMill}.
Let $L:=[a,b]$ be the ordered continuum constructed in \cite[Section 5]{Hvan}.
This space has very few continuous selfmaps. Any continuous map  $f\co L \to L$ is a \emph{canonical retraction}.
That is, there exists a pair $u \leq v \in L$ such that $$f(x)=u \ \forall x \leq u, f(x)=x \ \forall u \leq x \leq v, \ f(x)=v \ \forall x \geq v.$$
In particular, $L$ is $\H$-rigid. Moreover, for every 
topological embedding $f\co L \to L$ we have $f=id$. 
Note also the following \textit{special property} which we use below:
 if $f(a)=a$ then either $f(x)=x$ for every $x \in U$
on some neighbourhood $U$ of $a$, or $f$ is the constant map $f(x)=a$ for every $x \in L$.

Now the desired continuum $X$ will be the two point compactification of some locally compact connected LOTS $Y$,  the "long L". More precisely, the corresponding linearly ordered set $Y$ is the lexicographically ordered
set $\Z \times [a,b)$. Endow $Y$ with its usual interval topology. Every subinterval in $Y$ of the form
$$L_n:= [(n,a),(n+1,a)]=  \{(n,x): \ x \in [a,b)] \} \cup \{(n+1,a)\}$$ is naturally order isomorphic with $L$ for every $n \in \Z$.
Our aim is to show that $\H_+(X)=\H(X)=\Z$.
 First of all we have a naturally defined (shift) homeomorphism $\sigma\co X \to X$ where $\sigma(n,x)=(n+1,x)$ for every $n \in \Z, x \in [a,b)$.  We claim that any other homeomorphism
$f\co X \to X$ is $\sigma^k$ (the $k$-th iteration) for some $k \in \Z$. Indeed, if $f(L_0) \subseteq L_k$ for some $k \in \Z$ then
$f(L_0) = L_k$. Moreover it is easy to see that $f=\sigma^k$. Now assume that $f(L_0) \subseteq L_k$ is not true for every $k \in \Z$.  Then there exists $k \in \N$ such that $f(0,a) < (k,a) < f(0,b)$.
Consider the retraction $$h\co X \to X, h(z)=(k,a) \ \forall z \leq (k,a), \ \text{and} \ h(z)=z \ \forall z > (k,a).$$
Then the composition $h \circ f$ restricted on $L_0$ defines a nonconstant continuous map $L_0 \to L_k$ which moves $(0,a)$ to $(k,a)$. 
This induces a continuous nonconstant selfmap $q\co L \to L$ such that $q(x)=a$ \ \ $\forall \ x \in U$ for some neighborhood $U$ of $a$. 
 By the \emph{special property} 
 of $L$ mentioned above, we get a contradiction.
These arguments show that algebraically $\H_+(X) = \H(X) = \Z$. 
Finally observe that $\H_+(X)$ is discrete in the compact-open topology.

 \subsection{The Ordinal Space}

 For every ordinal number $\kappa$ the space $[0,\kappa]$ is a compact LOTS. 
 This space is scattered and hence not connected for every $\kappa >0$.
 Nonetheless, one can show that $\H_+[0,\kappa]$ is trivial (hence $a$-minimal). We start by noting that   $[0,\kappa]$ is certainly a well-ordered set.

 \begin{lemma}\textnormal{\cite[Corollary $4.1.9$]{Cies}}\label{uniqueIso}
 	If two well-ordered sets $A$ and $B$ are order-isomorphic, then the isomorphism is unique.
 \end{lemma}
 It follows from Lemma \ref{uniqueIso} that the identity is the only order-preserving automorphism of a well-ordered set.
 \begin{corol}\label{minOrdinal}
 	Every well-ordered compact LOTS $X$
 	(e.g., the ordinal space $X=[0,\kappa]$) is $\H_+$-rigid. That is, $\H_+(X)=\{e\}$  (thus $X$ is $aM_+$-compact).
 \end{corol}

  This example shows that the condition of Theorem \ref{t:new} is not necessary.

 \vskip 0.3cm

\subsection{The Unit Interval}
\label{s:int}

The group $\H_+[0,1]$ (and, hence, also any $\H_+[a,b]$ for every two reals $a <b$) is not abelian. Take, for example, the following pair $f,h$ of noncommuting elements.  Define $f(x)=x^2$, $h(x)=0.5x$ for $0 \leq x \leq 0.5$ and $h(x)=1.5x-0.5$ for $0.5 \leq x \leq 1$. 
 So, the continuum
$[0,1]$ clearly satisfies the conditions of Theorem  \ref{t:new}. 
Therefore, the groups $\H_+[0,1]$ and  $\H[0,1]$ are $a$-minimal.

\vskip 0.3cm

\subsection{The Ordered Square}
Let $I=[0,1]$ and define the lexicographic order on $I\times I$. Then $\mathcal I^2=\left( I\times I, \tau_{\leq}\right)$, the unit square with the order topology, is a compact and not metrizable space.
We show that it satisfies the conditions of 
Corollary \ref{c:suff}. It is connected (see \cite[Section $48$]{CEIT}). As to the second condition, let
 $K=[(a_1,b_1),(a_2,b_2)] \seq \mathcal I^2$ be a closed interval. If $a_1=a_2$ then $K$ is homeomorphic to $[0,1]\seq \mathbb R$. Otherwise, if $a_1<a_2$, $K$ contains an interval homeomorphic to $[0,1] \seq \mathbb R$
  (for example $[(\frac{a_1+a_2}{2},0),(\frac{a_1+a_2}{2},1)]$).
 Thus condition (C) of Corollary \ref{c:suff} is satisfied.
 Hence, $\H_+(\mathcal I^2)$ and $\H(\mathcal I^2)$ are $a$-minimal (and $\mathcal I^2$ is both $aM_+$-compact and $aM$-compact).

\subsection{The Extended Long Line}

\vskip 0.2cm

	Let $\mathcal{L}$ be the set $[0, \omega_1) \times [0,1)$ where  $\omega_1$ is the least uncountable ordinal.
	 Considering $\mathcal{L}$ with the lexicographic order, the set $\mathcal{L}$ with the topology induced
	by this order is called \textit{the long line}.
	Let $\mathcal{L}^* = \mathcal{L}\cup \{\omega_1 \}$ and extend the ordering on $\mathcal{L}$ to $\mathcal{L}^*$ by letting $a < \omega_1$ for all $a \in \mathcal{L}$. The space $\mathcal{L}^*$ with the order topology is a compact space called the \textit{extended long line}.
	In fact, $\mathcal{L}^*$ is the one point compactification of $\mathcal{L}$.

\vskip 0.2cm

 Several properties of this space can be found in \cite{Joshi}, \cite{Munkres} and \cite{Salzmann}. The extended long line satisfies the conditions of 
 Corollary \ref{c:suff}. Indeed, it is well known that $\mathcal L^*$ is a compact connected LOTS. Also, $\mathcal L$ (the long line) is locally homeomorphic (by an order-preserving homeomorphism) to the interval $(0,1)$. 
 In case the interval in question is of the form $[a,\omega_1]$, we can verify condition (C) for a subinterval $[a,b]$ of $[a,\omega_1]$, where $b \neq \omega_1$. 
So, $\H_+(\mathcal{L^*})$ and $\H(\mathcal{L^*})$ are $a$-minimal. Hence, $\mathcal{L^*}$ is both $aM_+$-compact and $aM$-compact.


\vskip 0.3cm

\subsection{The Circle}

Recall the definition of the natural cyclic ordering  (Definition \ref{d:cycord})
on the unit circle $\mathbb{S}^1$. Identify
$\mathbb{S}^1$, as a set, with $[0,1)$ and
define a ternary relation $R \subseteq [0,1)^3$ as follows:
$(z, y, x) \in R$ if and only if $(x-y)(y-z)(x-z) > 0$.
Denote by $\H_+(\mathbb{S}^1)$  the Polish group of all orientation preserving homeomorphisms of the circle $\mathbb{S}^1$.

The arguments of Theorem \ref{t:new} (or, of \cite[Theorem 1]{Gartside}) can be easily modified for the circle
$\mathbb{S}^1,$ hence:

\begin{thm}\label{H+S1}
	The group $\H_+(\mathbb{S}^1)$ is $a$-minimal.
\end{thm}

Note that the coset space $\H_+(\mathbb{S}^1) / \St(z)$ is naturally homeomorphic to the circle, where $\St(z)$ is the stabilizer group of any given $z \in \mathbb{S}^1$. So the minimality of $\H_+(\mathbb{S}^1)$
can be derived from the minimality of $\H_+[0,1]$ using Lemma \ref{general_lemma} and the fact that $\St(z)$ is topologically isomorphic to $\H_+[0,1]$. 

Since $\H_+(\mathbb{S}^1)$ is a closed normal subgroup of $\H(\mathbb{S}^1)$, and $\H(\mathbb{S}^1)/\H_+(\mathbb{S}^1)\cong \mathbb Z_2$, we can use Lemma \ref{general_lemma} one more time to deduce the minimality of $\H(\mathbb{S}^1)$. 

A Hausdorff topological group is \textit{totally minimal} if every Hausdorff quotient is minimal \cite{DikTotalMin}. Every minimal algebraically (or, at least, topologically) simple minimal group is totally minimal. $\H_+(\mathbb{S}^1)$ is algebraically simple as can be seen (for example) in \cite{SU,Ghys}.
Although the group $\H_+[0,1]$ is not algebraically simple, it is topologically simple. Indeed, by \cite[Theorem 14]{FS}, $\H_+[0,1]$ has exactly five normal subgroups: $\{e\}$, $\H_+[0,1]$, $Q_1$, $Q_0$, $Q:=Q_0 \cap Q_1$. It is easy to see that
 $Q$ is dense in $\H_+[0,1]$. This yields that $\H_+[0,1]$ is topologically simple.

\begin{corol} 
	$\H_+(\mathbb{S}^1)$ and $\H_+[0,1]$ are totally minimal groups.
\end{corol}

\section{Some questions}

A more general version of Question \ref{q:M} is the following.

\begin{question} \label{q:3} 	
	When appropriate subgroups $G$ of $\H(X)$ (say, the automorphism groups of some structures on $X$) are minimal (a-minimal) ?
\end{question}

We already know that the Cantor cube $2^{\omega}$ is $M$-compact (\cite{Gam}).

\begin{question} \label{q:2} \
	\begin{enumerate}
		\item Is the Cantor cube $2^{\omega}$ $aM$-compact ?
		\item Is the Cantor set $X \subseteq [0,1]$, as a linearly ordered compact LOTS, $M_+$-compact ?  $aM_+$-compact ?
		\item Is the space $2^{\lambda}$ $M$-compact (or, $aM$-compact) for every cardinal $\lambda$ ?
	\end{enumerate}
\end{question}

\begin{question} \label{q:1} \
\begin{enumerate}
	\item Is it true that every $M$-compact space is also $aM$-compact ?
	\item Is it true that every linearly ordered connected $M_+$-compact space is $aM_+$-compact ?
	\item Is it true that for ordered continua condition (A) of Theorem \ref{t:new} is really weaker than condition (C) of Corollary \ref{c:suff} ?
\end{enumerate}	
\end{question}

In view of Markov's Question \ref{q:Markov} and Theorem \ref{t:new}
we have several good reasons to pose the following question.

\begin{question} \label{q:Markov2}
	For what compact (linearly ordered) spaces $X$ the Markov and Zariski topologies coincide on the group $G=\H(X)$  (resp., $G=\H_+(X)$) ?
\end{question}

Various properties of the homeomorphism group $\H(X)$ of several important 1-dimensional continua $X$  were intensively studied from several points of view. Among others is the case where $X$ is the \emph{pseudo-arc} or the \emph{Lelek Fan}. About the latter case, see, for example, the very recent works of Barto\v{s}ova-Kwiatkowska \cite{BartKw,BartKw2}  and Ben Yaacov-Tsankov \cite{BeTs}.

\begin{question} \label{q:4} 	
	Let $X$ be the pseudo-arc or the Lelek fan. Is it true that $\H(X)$ is minimal ? $a$-minimal ?
\end{question}

It is well known that the pseudo-arc is a homogeneous compactum. So the previous question is related to Stoyanov's Question \ref{q:St}. Another property of the pseudo-arc is that it is a chainable continuum.
Recall that a compact space $X$ is \emph{chainable} if every (finite) open cover $\eps$ has a
finite open refinement $\alpha$ that is an  \emph{$\eps$-small chain}, that is,  $\alpha=\{O_1, \dots, O_n\}$, where
$O_i \cap O_j \neq \emptyset \Leftrightarrow |i-j| \leq 1$ and every $O_i$ is $\eps$-small.
Every linearly ordered continuum is chainable. This follows, for example, by Lemma \ref{l:connected}.
Therefore, it would be interesting to extend Theorem \ref{t:new} to some broader class of chainable continua.





\vskip 2cm
\begin{thebibliography}{00}






\bibitem[Arhangelskii(1987)]{Arhan}
A.V. Arhangel'ski\u{i}, \textit{Topological homogeneity, topological groups and their continuous images}, (Russian) Uspekhi Mat. Nauk \textbf{42} (1987), no. 2 (254), 69-105. 


\bibitem[Banakh, Guran, Protasov(2012)]{BGP}
T. Banakh, I. Guran, I. Protasov, \textit{Algebraically determined topologies on permutation groups}, Topology
Appl., Special Issue on Dikran Dikranjan's 60th Birthday. \textbf{159}  (2012), 2258-2268.

\bibitem{BartKw}
D. Barto\v{s}ova, A. Kwiatkowska, \emph{Lelek fan from a projective Fraïssé limit}, 2013, arXiv:1312.7514.


\bibitem{BartKw2}
D. Barto\v{s}ova, A. Kwiatkowska, \emph{Gowers' Ramsey Theorem with multiple operations and dynamics of the homeomorphism group of the Lelek fan}, 2015, arXiv:1411.5134.


\bibitem{BeTs}
I. Ben Yaacov, T. Tsankov,
\emph{Weakly almost periodic functions, model-theoretic stability, 
and minimality of topological groups}, 2013, arXiv:1312.7757.




\bibitem{Cech}
E. \v{C}ech, \emph{Point Sets}, Academia, Prague, 1969.


\bibitem[Ciesielski(1997)]{Cies}
K. Ciesielski, \textit{Set Theory for the Working Mathematician}, Cambridge University Press, 1997.

\bibitem[Comfort, Hofmann, Remus (1992)]{CHR}
W.W. Comfort, K.H. Hofmann, D. Remus, \textit{A survey on topological groups and semigroups}, in: M Husek
and J. van Mill, eds. Recent Progress in General Topology, North Holland, Amsterdam, 1992, 58-144.

\bibitem[Dierolf, Schwanengel(1977)]{DS}
S. Dierolf, U. Schwanengel, \textit{Un example d'un groupe topologique q-minimal mais non precompact}, Bull.
Sci. Math. \textbf{101} (1977), 265-269.

\bibitem[Dikranjan(1998)]{Dikr}
D. Dikranjan, \textit{Recent advances in minimal topological groups}, Topology Appl., {\bf 85} (1998) 53-91.

\bibitem[Dikranjan, Megrelishvili(2014)]{MegDik}
D. Dikranjan, M. Megrelishvili, \textit{Minimality Conditions in Topological Groups}, in: Recent progress in general topology III, 229-327, K.P. Hart, J. van Mill, P. Simon (Eds.), Springer, Atlantis Press, 2014.

\bibitem[Dikranjan, Prodanov, Stoyanov(1989)]{DPS}
D. Dikranjan, Iv. Prodanov, L. Stoyanov, \textit{Topological groups: characters, dualities and minimal group topologies}, Pure and Appl. Math. \textbf{130}, Marcel Dekker, New York-Basel, 1989.

\bibitem[Dikranjan, Prodanov(1974)]{DikTotalMin}
D. Dikranjan, Iv. Prodanov, \textit{Totally minimal groups}, Ann. Univ. Sofia Fac. Math. M´ec. \textbf{69} (1974/75) 5-11.

\bibitem[Dikranjan, Shakhmatov(2010)]{DikSh10}
D. Dikranjan, D. Shakhmatov,
\emph{The Markov–Zariski topology of an abelian group},
Journal of Algebra, {\bf 324} (2010) 1125-1158.

\bibitem[Dikranjan, Shakhmatov(2007)]{DikSh}
D. Dikranjan, D. Shakhmatov, \textit{Selected topics from the structure theory of topological groups}, in: Open Problems in Topology,
II (E.Pearl ed.), Elsevier, 2007, 389-406.




\bibitem[Dikranjan, Toller(2012)]{DikTol}
D. Dikranjan, D. Toller, \textit{Markov’s problems through the looking glass of Zariski and Markov topologies, Ischia Group Theory},
2010, Proc. of the Conference, World Scientific Publ. Singapore, 2012, 87-130.

\bibitem[Do\"{\i}tchinov(1972)]{Doitch}
D. Do\"{\i}tchinov,\textit{Produits de groupes topologiques minimaux}, Bull. Sci. Math. (2) \textbf{97} (1972), 59-64.

\bibitem[Eberhardt, Dierolf, Schwanengel(1980)]{Eberhardt}
V. Eberhardt, S. Dierolf, U. Schwanengel, \textit{On products of two (totally) minimal topological groups and the three-space-problem}, Math. Ann., \textbf{251} (1980), 123-128.


\bibitem[Engelking(1989)]{Engel}
R. Engelking, \textit{General topology}, revised and completed edition, Heldermann Verlag, Berlin, 1989.


\bibitem[Fine,Schweigert(1955)]{FS}
N.J. Fine, G.E. Schweigert,
{\em On the group of homeomorphisms of an arc},
Ann. of Math. \textbf{62} (1955), 237-253.


\bibitem[Gamarnik(1991)]{Gam}
D. Gamarnik, \textit{Minimality of the Group $\Aut(C)$}, Serdika \textbf{17} (1991), no. 4, 197-201.

\bibitem[Gartside, Glyn(2003)]{Gartside}
P. Gartside, A. Glyn, \textit{Autohomeomorphism groups}, Topology Appl. \textbf{129}
(2003), 103-110.

\bibitem[Gaughan(1967)]{Gaughan}
E.D. Gaughan, \textit{Topological group structures of infinite symmetric groups}, Proc. Nat. Acad. Sci. U.S.A. \textbf{58} (1967), 907-910.

\bibitem[Ghys(2001)]{Ghys}
E. Ghys, \textit{Groups acting on the circle},  Enseign. Math. (2), \textbf{47} (2001), 329-407.

\bibitem[Glasner, Megrelishvili(2014)]{GM-tame}
E. Glasner, M. Megrelishvili, \textit{Eventual nonsensitivity and tame dynamical systems}, 2014,
arXiv:1405.2588.


\bibitem{Hvan}
K.P. Hart, J. van Mill, \emph{A method for constructing ordered continua}, Topology Appl., {\bf 21} (1985), 35-49.

\bibitem[Joshi(1983)]{Joshi}
K.D. Joshi, \textit{Introduction to general topology}, John Wiley \& Sons Canada, Limited, 1983.

\bibitem[]{Kok}
H. Kok, \emph{Connected orderable spaces},
Math. Centhre Tracts {\bf 49}, Mathematisch Centrum, Amsterdam, 1973.



\bibitem[Lukach(2012)]{Lukach}
G. Luk\'{a}cs, \textit{Report of the Open Problems Session}, Topology Appl.
 \textbf{159} (2012), 2476-2482.


%
%

\bibitem[van Mill(2012)]{vanMill2012}
J. van Mill, \textit{Homeomorphism groups of homogeneous compacta need not be minimal}, Topology Appl., Special Issue on Dikran Dikranjan's 60th Birthday. \textbf{159} (2012), 2506-2509.

\bibitem[Munkres(2000)]{Munkres}
J. Munkres, \textit{Topology} (2nd Edition), Pearson, 2000.

\bibitem[Nachbin(1965)]{Nachbin}
L. Nachbin, \textit{Topology and Order}, Princeton, NJ: VanNostrand Co, 1965.

\bibitem[Pestov(2006)]{Pestov2006}
V. Pestov, \textit{Dynamics of Infinite-dimensional Groups: the Ramsey-Dvoretzky-Milman Phenomenon}, Providence, R.I.: American Mathematical Society, 2006.


\bibitem[Roelcke, Dierolf(1981)]{Roelcke}
W. Roelcke, S. Dierolf, \textit{Uniform structures on topological groups and their quotients}, McGraw-Hill, New York, 1981.

\bibitem[Salzmann(2007)]{Salzmann}
H. Salzmann, et al. \textit{The Classical Fields}, 
 Cambridge University Press, 2007.

\bibitem[Schreier,Ulam(1934)]{SU}
J. Schreier, S. Ulam, {\em Eine Bemerkung uiber die Gruppe der topologischen Abbildungen der Kreislinie auf sich selbst}, Studia Math., \textbf{5} (1934), pp. 155-159.

\bibitem[Steen,Seebach(1995)]{CEIT}
 L.A. Steen, J.A. Seebach, \textit{ Counterexamples in Topology}, Dover, 1995.

\bibitem[Stephenson(1971)]{Steph}
R. Stephenson, \textit{Minimal topological groups}, Math. Ann., \textbf{192} (1971), 193-195.

\bibitem{Tod}
S. Todorcevic, \emph{Trees and linearly ordered sets},
in: Handbook of Set-Theoretic Topology, eds: K. Kunen, J.E. Vaughan, NH, 1984.


\bibitem[Uspenskij(2001)]{Usp2001}
V.V. Uspenskij, \textit{The Roelcke compactification of groups of homeomorphisms}, Topology Appl. \textbf{111} (2001), no. 1-2, 195-205.

\bibitem[Uspenskij(2008)]{Usp2008}
V.V. Uspenskij, \textit{On subgroups of minimal topological groups}, Topology Appl. \textbf{155} (2008), 1580-1606.


\end{thebibliography}



\end{document}